\documentclass{article}

\usepackage[preprint]{neurips_2020}

\usepackage{natbib}
\usepackage[utf8]{inputenc} 
\usepackage[T1]{fontenc}    
\usepackage{hyperref}       
\usepackage{url}            
\usepackage{booktabs}       
\usepackage{amsfonts}       
\usepackage{nicefrac}       
\usepackage{microtype}      
\usepackage{graphicx}
\usepackage{subfigure}
\usepackage{booktabs} 
\usepackage[utf8]{inputenc}
\usepackage[utf8]{inputenc}
\usepackage{amsmath, amssymb, amsfonts, amsthm}
\usepackage{graphicx}
\usepackage{color}
\usepackage{bm}
\graphicspath{ {figs/} }
\usepackage{mathabx}
\usepackage{algorithm}
\usepackage[noend]{algpseudocode}
\newtheorem{theorem}{Theorem}
\newtheorem{claim}{Claim}
\usepackage{hyperref}
\usepackage{todonotes}

\title{Learning Expected Reward for Switched Linear Control Systems: A Non-Asymptotic View}

\author{%
  Muhammad Abdullah Naeem  \\
  Department of Electrical and Computer Engineering\\
  Duke University\\
  Durham, NC 27708 \\
  \texttt{muhammad.abdullah.naeem@duke.edu}\\
\And
  Miroslav Pajic\\
  Department of Electrical and Computer Engineering\\
  Duke University\\
  Durham, NC 27708 \\ 
  miroslav.pajic@duke.edu \\
}

\begin{document}

\maketitle

\begin{abstract}
\vspace{-8pt}
In this work, we show existence of invariant ergodic measure for switched linear dynamical systems (SLDSs) under a norm-stability assumption of system dynamics~in~some unbounded subset of $\mathbb{R}^{n}$. Consequently, given a stationary Markov control policy, we derive non-asymptotic bounds for learning expected reward (w.r.t the invariant ergodic measure our closed-loop system mixes to) from time-averages using Birkhoff's Ergodic Theorem. The presented results provide a foundation for deriving non-asymptotic analysis for average reward-based optimal control of SLDSs. Finally, we illustrate the presented theoretical results in two case-studies.
\end{abstract}

\section{Introduction}
\label{submission}
Last decade has seen tremendous advancements in non-asymptotic analysis of system identification and optimal control for linear time-invariant dynamical systems (e.g.,~\cite{tu2018least,hao2020provably, oymak2019stochastic,fazel2018global,simchowitz2018learning,sarkar2019finite}). 
When evaluating a value function corresponding to a policy for the infinite-horizon case, existing literature in non-asymptotic analysis, such as~\cite{lazaric2012finite,tu2018least}, uses a discount factor on instantaneous rewards;\footnote{In this work, by the reward function we actually refer to the cost function in the control sense.}  
this  results in a policy iteration algorithm that computes a new policy that minimize immediate rewards rather than minimizing the cumulative reward over infinite horizon. 
While this approach might be valid for some application domains (e.g.,  finance), it may not be suitable in general setting; 
in the general case, it is preferable to minimize the expected reward with respect to the stationary distribution induced by the choice of the control policy~\cite{bertsekas1995dynamic}.  

The main challenge in using this approach is that, when system dynamics is not known, we do not have access to the stationary distribution.~However, let us assume that we do have access~to~instantaneous rewards and are able to show that the underlying dynamical system mixes geometrically to an invariant ergodic measure. Then, the expected reward w.r.t the stationary distribution induced by our choice of control policy can be approximated from time-averages of instantaneous rewards, under mild assumptions on the reward function. On the other hand, when system dynamics is unknown, average reward-based optimal control requires bounds on the mixing time; i.e, these methods require a bound on the length of time-averages of the reward that, with high probability, approximate the expected reward w.r.t the stationary distribution induced by the choice of the control policy. \cite{zahavy2019average} recently provided non-asymptotic analysis for this problem when the underlying state-space is discrete. 
In this work, we focus on the case when the controlled dynamical systems is an unknown switched linear dynamical system in continuous state-space~$\mathbb{R}^{n}$. We first show sufficient conditions for existence of an ergodic invariant distribution. After establishing existence, we provide analysis for non-asymptotic sample complexity of learning the expected reward from its time-averages.

\subsection{Preliminaries and Background}
 \vspace{-2pt}
\paragraph{Notation.} $\mathbb{N} $ and $\mathbb{R}$ denote the sets of natural and real numbers, respectively.  ${I}_{n}\in\mathbb{R}^{n\times n}$ denotes the $n$ dimensional identity matrix, whereas $\| \nu -\mu \|_{tv}$ is the total variation distance~between probability measures $\mu$ and $\nu$. 
For random variables $x$ and $y$, $\mathbb{E}(x)$ and $Cov(x,y)$ denote the expectation and covariance. 
$\lambda^{n}()$ is the n-dimensional Lebesgue measure, and $\mathbb{B}_{\mathbb{R}^{n}}$ is Borel $\sigma$-algebra on $\mathbb{R}^{n}$. $ \mathcal{B}_{\alpha}^{n}:=\{x \in \mathbb{R}^{n}: \|x\|_2
\leq \alpha \}$ is the 
$\alpha$-ball in $\mathbb{R}^n$. Also, $y_n \xrightarrow{a.s} y$ denotes that $y_n$ converges almost surely to $y$, whereas to simplify our presentation $a \land b$ denotes $\min(a,b)$. $\chi_{\{\}}()$ is the indicator function,  $Q \succ  0$ ($Q \succeq  0$) denotes that matrix $Q$ is positive (semi)definite, and $\rho(A)$ is the spectral radius of matrix $A \in \mathbb{R}^{n \times n}$. Finally, for a set $\mathcal{K}\subseteq\{1,...,M\}$, its complement is $\mathcal{K}^\complement:=\{1,...,M\}\setminus\mathcal{K}$. 

\paragraph{Background.} We consider a discrete-time switched linear dynamical system (SLDS) of the form 
\vspace{-2pt}
\begin{equation}
\label{eq:sls}
x_{t+1}=\sum_{j=1}^{M} (A_j x_t+B_J u_t +w_{t}^{j})\chi_{\mathcal{M}_j} (x_t).
\end{equation}
Here, $x_t \in \mathbb{R}^n, u_t \in \mathbb{R}^p$ denote the system's state and input, respectively, and $A_j \in \mathbb{R}^{n \times n}$ and $B_j \in \mathbb{R}^{n \times p}$, $j=1,...,M$, capture system dynamics in each of the $M$ regions that decompose the state-space -- the regions, defined as $\mathcal{M}_j = \{x \in \mathbb{R}^{n}: L_j x \leq C_j\}$, $j=1,...,M$, are pairwise disjoint satisfying $\bigcup_{j=1}^{M} \mathcal{M}_j = \mathbb{R}^{n}$.\footnote{To simplify our notation, we employ a polyhedra-based region representation. On the other hand, any pairwise-disjoint region decomposition of the state space would suffice.}
In addition, for a fixed region $j$, noise vectors $w_{t}^{j} $ are i.i.d, and satisfy $w_{t}^{j} \thicksim \mathcal{N}(0,{I}_{n})$ 
and 
$Cov(w_{t}^j,w_{s}^k)=0$, 
for all $t,s\geq 0$ and $j \neq k \in \{1,2,...,M\}$.  

We assume that the applied control law $u_t$ is a linear function of state $x_t$ weighted by policy $\pi$ -- i.e., $u_t= \pi x_t$, 
with $\pi\in\mathbb{R}^{p \times n}$. We use a common (control) reward function $r(x,u)= \sqrt{x^{T}Qx+u^{T} R u}$, where
$Q \succcurlyeq 0 , R \succ  0$, and $Q,R \in  \mathbb{R}^{n \times n}$~\cite{bertsekas1995dynamic}. 
Hence, under the control law  $u=\pi x$, which we also denote as $u=\pi (x)$, we have that
$r(x,\pi x)= \sqrt{x^{T}(Q+\pi^{T}R \pi)x}$;
furthermore, the closed-loop dynamics of~\eqref{eq:sls} 
can be captured as
 \vspace{-2pt}
\begin{equation}
\label{eq:slcs}
x_{t+1}=\sum_{j=1}^{M} \big(\underbrace{(A_j +B_J\pi)}_{\hat{A}_j^{\pi}}x_t + w_{t}^{j}\big)\chi_{\mathcal{M}_j} (x_t)=\sum_{j=1}^{M} \left(\hat{A}_j^\pi x_t  +w_{t}^{j}\right)\chi_{\mathcal{M}_j} (x_t).
\end{equation} 
Finally, if $x_t$ from~\eqref{eq:slcs} under policy $\pi$ mixes to a stationary distribution $\nu_{\pi}$, we define the steady-state reward $\rho(\pi)$ associated with the policy $\pi$ as $\rho(\pi) = \mathbb{E}_{ x \thicksim \nu_{\pi}} r(x, \pi(x))$.

The contributions of this paper are twofold. 
First, we derive sufficient conditions under which samples of the closed-loop system trajectories from~\eqref{eq:slcs}, under policy $\pi$, mix geometrically to a unique ergodic invariant measure $\nu_{\pi}$ (in Section~\ref{sec:mixing}). 
Second, when the closed-loop system in~\eqref{eq:slcs} satisfies the derived conditions, leveraging Birkhoff's pointwise ergodic theorem, which implies
\begin{equation}
\label{eq:rho}
\lim_{N \rightarrow \infty} \frac{1}{N} \sum\nolimits_{t=0}^{N-1} r(x_t, \pi(x_t)) \xrightarrow{a.s}  \int r(y,\pi(y)) \nu_{\pi} (dy) =: \rho(\pi),
\end{equation}
we provide finite sample analysis for learning $\rho(\pi)$, defined in~\eqref{eq:rho}, with high probability from time averages $ \frac{1}{N} \sum_{t=0}^{N-1} r(x_t, \pi(x_t))$ (in Section~\ref{sec:samples}).\footnote{Since 
$r(x,\pi x)= \sqrt{x^{T}(Q+\pi^{T}R \pi)x}$ is effectively a function of $x$, to simplify our notation, from now on we will use $r(x_t)$ instead of $r(x_t, \pi(x_t))$; the role of policy $\pi$ will also  always be clear from $\nu_{\pi}$.} 
We show that the complexity of the sample analysis linearly depends on the size of the state space as opposed to quadratic dependence when a discounted LQR is used on commonly considered linear Gaussian dynamical systems (e.g.,~as in \cite{tu2018least}). 
Finally, we validate the presented  results on case studies (in Section~\ref{sec:case}).



\section{Mixing of SLDSs to Invariant Distributions}
\label{sec:mixing}

We start by considering a  linear time-invariant 
dynamical system with a state-space representation 
\begin{equation}
\label{eq:LDS}
    x_{t+1}= Ax_t +w_t,
\end{equation}
where $(w_t)_{t \in \mathbb{N}}$ is i.i.d with distribution $\mathcal{N}(0,{I}_n)$ and spectral radius $\rho(A)$ satisfies that $\rho(A)<1$. For $\mathcal{A} \in \mathbb{B}_{\mathbb{R}^{n}}$ and any $s \in \mathbb{N}$, one step transition kernel is defined as $P(x,\mathcal{A})= \mathbb{P}(x_{s+1} \in \mathcal{A} | x_s=x)$ and $k^\text{th}$ step transition kernel is denoted by $P^{k}(x,\mathcal{A})= \mathbb{P}(x_{s+k} \in \mathcal{A} | x_s=x)$.  
Proving existence of an ergodic invariant measure using a total variation approach, would require showing that for all $x,y \in \mathbb{R}^{n}$ it holds that $\lim_{t \rightarrow \infty} \|P^{t}(x,\cdot)-P^{t}(y,\cdot)\|_{tv} = 0$. 

For SLDS~\eqref{eq:LDS}, $P(x,\cdot)=\mathcal{N}(Ax, I_{n})$, and it is a common knowledge that the sequence from~\eqref{eq:LDS} mixes geometrically to a unique ergodic invariant Gaussian distribution~\cite{tu2018least}; at the same time $\sup_{(x,y) \in \mathbb{R}^{n} \times \mathbb{R}^{n} } \|P(x,\cdot)-P(y,\cdot)\|_{tv}=2$, making total variation approach infeasible for Gaussian kernel on unbounded state space. Adding to the difficulty of SLDS analysis, the transition kernel for the state sequence of the closed-loop system~\eqref{eq:slcs} is more complex than for linear time-invariant systems~\eqref{eq:LDS}, which is a standard benchmark 
in non-asymptotic analysis(e.g.,~\cite{abbasi2019model,tu2018least}).

This brings us to the theory of  \emph{Wasserstein metric and optimal transport}~\cite{villani2008optimal}, which is used to construct a metric on $\mathcal{P}(\mathbb{R}^n)\times\mathcal{P}(\mathbb{R}^n)$ under which~\eqref{eq:slcs} mixes to an invariant measure; here, $\mathcal{P}(\mathbb{R}^n)$ is the space of probability measures on $\mathbb{R}^n$. For a lower semi-continuous metric $d(x,y)$ on $\mathbb{R}^{n} \times \mathbb{R}^{n}$, Wasserstein metric on $\mathcal{P}(\mathbb{R}^n)$ is defined~as   
\begin{equation}
\label{eq:WM}
    \mathcal{W}_{d}^{1} (\nu,\mu)= \inf_{(X,Y) \in \Gamma(\nu,\mu)} \mathbb{E}~d(X,Y);
\end{equation}
here, $(\nu,\mu) \in \mathcal{P}(\mathbb{R}^n) \times \mathcal{P}(\mathbb{R}^n)$, and $(X,Y) \in \Gamma(\nu,\mu)$ implies that random variables $(X,Y)$ follow probability distributions on $\mathbb{R}^{n} \times \mathbb{R}^{n}$ with marginals $\nu$ and $\mu$. For example, if $d(x,y):= \chi_{x \neq y}(x,y)$, it follows that 
\begin{equation}
\label{eq:Wtv}
\mathcal{W}^{1}(\nu,\mu)= \inf_{(X,Y) \in \Gamma(\nu,\mu)} \mathbb{P}(X \neq Y)= \frac{1}{2}\|\nu-\mu\|_{tv}.   
\end{equation}

For any function $V: \mathbb{R}^{n} \rightarrow [0,\infty)$, we define 
$\mathcal{P}V(x):=\int_{\mathbb{R}^n}V(y)P(x,dy)$.\footnote{Notice that although we abuse the notation and use $\mathcal{P}$ to represent both the transition operator and the space of probability measures, its purpose and specific meaning will always be unambiguous from the context.} Intuitively speaking, to prove existence of an invariant measure, our goal is to show that for all $x,y \in \mathbb{R}^{n}$ their exists some metric $d_{\mathcal{P}}$ on $\mathcal{P}(\mathbb{R}^n) \times \mathcal{P}(\mathbb{R}^n)$ such that $d_{\mathcal{P}}$ acts as a contraction on the transition kernel of~\eqref{eq:slcs} --  i.e.,
\begin{equation}
\label{eq:t6}
d_{\mathcal{P}}\left(P^{t}(x,\cdot),P^{t}(y,\cdot)\right) \leq \eta\cdot d_{\mathcal{P}}\left(P^{t-1}(x,\cdot),P^{t-1}(y,\cdot)\right),
\end{equation}
for some $\eta<1$ and all $t \in \mathbb{N}$. If the space $\mathcal{P}(\mathbb{R}^n)$ is complete under the metric $d_{\mathcal{P}}$,~\eqref{eq:t6} implies existence of a unique invariant measure that the SLDS from~\eqref{eq:slcs} mixes to.~\cite{hairer2010convergence} introduces corresponding easy-to-verify conditions; if there exist function $V(x), \gamma \in (0,1), K \in \mathbb{R}$, and $\alpha > 0$~s.t.
%
    \begin{align}
    \label{eq:EIcond1}
    (i)&\hspace{10pt} \mathcal{P}V(x) \leq \gamma V(x)+ K,~\hspace{50pt}\text{for all }~x \in \mathbb{R}^{n}\\
    \label{eq:EIcond2}
    (ii)&\hspace{10pt}  \|P(x,\cdot)-P(y,\cdot)\|_{tv} \leq 2(1-\alpha),\hspace{8pt} \text{for all }~{(x,y) \in \mathbb{R}^{n} \times \mathbb{R}^{n}~s.t.~ V(x)+V(y) \leq \hat{r} },
    \end{align}
    with $\hat{r} > \frac{2K}{1-\gamma}$,
then system~\eqref{eq:slcs}  mixes geometrically to a unique ergodic invariant measure.\footnote{$V$ can be thought of as a Lyapunov function of the dynamical system; Lyapunov functions are widely used in control theory to capture energy of the system in a way that facilitates reasoning about system stability~\cite{van2007stochastic}.} 

In the above condition,~\eqref{eq:EIcond1} ensures that the dynamical system is being `pushed' to a neighborhood of the origin in $\mathbb{R}^{n}$.~\eqref{eq:EIcond2} implies existence of a sufficiently large level set such that any two trajectories, with distinct initial conditions inside the level set, can be `coupled' together with positive probability; i.e., 
if $X \thicksim P(x,\cdot)$ and $Y \thicksim P(y,\cdot)$ for $x,y$ in the aforementioned sufficiently large level set, then $\mathbb{P}(X \neq Y) \leq (1-\alpha)$ (see \cite{kulik2015introduction} for more details) and $\mathbb{P}(X_{n} \neq Y_{n}) \leq (1-\alpha)^n$. This implies that distinct trajectories overlap with positive probability and it becomes more and more unlikely as $n$ increases that they diverge away from each other. 
As shown in~\cite{hairer2010convergence,eberle2015markov} ,~\eqref{eq:EIcond1} and~\eqref{eq:EIcond2} ensure existence of a probability metric which acts as a contraction on the transition kernel, i.e., ~\eqref{eq:t6} holds for $ d_\mathcal{P}:=\mathcal{W}_{d_{{\beta}^\ast}}^{1}$, as defined in~\eqref{eq:WM}, where $d_{\beta^\ast}(x,y):=(2+\beta^\ast V(x)+ \beta^\ast V(y))\cdot\chi_{x \neq y}(x,y)$ for some $\beta^\ast>0$ and $\eta < 1$; note here that $\eta$ depends on $\alpha$ and $\beta^\ast$. 
Furthermore, completeness of $\mathcal{P}(\mathbb{R}^{n})$ equipped with metric $\mathcal{W}_{d_{\beta^\ast}}^{1}$ directly follows from lower semi-continuity of $d_{\beta^\ast}(x,y)$.

\begin{theorem}
\label{thm:t1}
Consider a control law $u_t=\pi(x_t)$, and assume that there exists $ \varrho< \infty$ 
such that for all $k_{\text{unbdd}} \in \mathcal{K}_{unbdd}:= \left\{ k ~|~ ( 1\leq k \leq M)  \text{ and } 
\left((\mathcal{M}_k \cap (\mathcal{B}^n_{\varrho})^\complement)
\neq \emptyset \right) \right\}$, it holds that $\| \hat{A}_{k_{\text{unbdd}}}^\pi \|_{2}<1$. Then,~the system \eqref{eq:slcs} mixes geometrically to a unique ergodic invariant distribution~$\nu_{\pi}$.
\end{theorem}

Before proving Theorem~\ref{thm:t1}, note that the theorem assumption is that, given a control law $u_t=\pi(x_t)$ there exists a bounded ball around the origin where the closed-loop dynamics might be unstable, but is  stable outside the ball. 
The bounded set is essentially same as the aforementioned `sufficiently large level set', and to show that~\eqref{eq:EIcond1} and \eqref{eq:EIcond2} hold, we proceed as follows.
\begin{proof}
Consider function $V(x)=\|x\|_{2}^2$. From~\eqref{eq:slcs}, we have $P(x,\mathcal{A})= \sum_{j=1}^{M}P_j(x,\mathcal{A})\chi_ {\mathcal{M}_j}(x)$, where $P_j(x,\cdot) \thicksim \mathcal{N}(\hat{A}_j ^{\pi} x, {I}_n)$. 
We define $c := \max_{k \in \mathcal{K}_{unbdd}^\complement}
\|\hat{A}_k ^{\pi}\|_2 ^{2}$ 
and $\gamma:=\max_{j \in \mathcal{K}_{unbdd} } \|\hat{A}_j ^{\pi}\|_2 ^{2}$; then, from the theorem assumption it holds that $\gamma < 1 $.

If we  assume that the initial state  $x_0:=x$ satisfies $\|x\|_{2} \leq \varrho$, then  there exists a $k \in \mathcal{K}_{unbdd}^\complement$  such that 
\begin{equation}
\label{eq:t9}
 \mathcal{P}V(x)= \mathbb{E}_{y \thicksim \mathcal{N}(\hat{A}_k ^{\pi}x, I_n)} \|y\|_{2} ^{2}= \mathbb{E}_{z \thicksim \mathcal{N}(0,I_n)} \|z\|_2 ^{2} + \|\hat{A}_k ^{\pi}x\|_{2} ^{2} \leq (n+c\varrho^{2}).
\end{equation}
However, if the initial state is $x_0:=x$ such that $ \|x\|_{2} > \varrho$, then there exists $j \in  \mathcal{K}_{unbdd}$  such~that
\begin{equation}
\label{eq:t10}
\mathcal{P}V(x)= \mathbb{E}_{y \thicksim \mathcal{N}(\hat{A}_j ^{\pi}x, I_n)} \|y\|_{2} ^{2}= \mathbb{E}_{z \thicksim \mathcal{N}(0,I_n)} \|z\|_2 ^{2} + \|\hat{A}_j ^{\pi}x\|_{2} ^{2} \leq n+\gamma \|x\|_{2}^{2}=(n+\gamma V(x)).
\end{equation}

Therefore, starting from any initial condition in $\mathbb{R}^n$, from~\eqref{eq:t9} and~\eqref{eq:t10} it holds that 
\begin{equation}
\label{eq:t11}
\mathcal{P}V(x) \leq \gamma V(x) + \underbrace
{(n+c\varrho^{2})}_{K},
\end{equation}
and thus~\eqref{eq:EIcond1} holds for the closed-loop dynamical system from~\eqref{eq:slcs} under the theorem assumptions. 

To show~\eqref{eq:EIcond2}, we define 
\begin{equation}
\label{eq:r}
    \hat{r}:= \frac{2(n+c\varrho^{2})}{\gamma(1-\gamma)} \qquad \Rightarrow \qquad \hat{r} > \frac{2(n+c\varrho^{2})}{(1-\gamma)},
\end{equation}
where the right side holds since $\gamma <1$. 
For any $x,y \in \mathbb{R}^{n}$, $P(x,\cdot)$ and $P(y,\cdot)$ follow $\mathcal{N}(\hat{A}_j ^{\pi}x, {I}_n)$ and $\mathcal{N}(\hat{A}_k ^{\pi}y, {I}_n)$ respectively, for some $j, k \in \{1,2,..,M\}$. 
Gaussians are absolutely continuous w.r.t. Lebesgue measure, implying that $\|P(x,\cdot)-P(y,\cdot)\|_{tv}= 2- 2 \int_{\mathbb{R}^n} (f_{x}(z) \land g_{y}(z))dz$ (see e.g.,~\cite{van2014probability}), with $f_{x}(z)$ and $g_{y}(z)$ being the density functions of $P(x,\cdot)$ and
$P(y,\cdot)$, respectively. Now, let us define 
 $$\alpha_{(x,y)}:=\int_{\mathbb{R}^n} (f_{x}(z) \land g_{y}(z)) dz$$
We have that $\alpha_{(x,y)}=0 \iff  f_{x}(z) \land g_{y}(z) =0 $
almost everywhere (a.e.) w.r.t. Lebesgue measure on $\mathbb{R}^n$~\cite{folland2013real}. Thus, the existence of $\alpha_{(x,y)} > 0$, 
for any specific $(x,y)$, directly follows if we can show that $f_{x}(z) \land g_{y}(z)>0 $ a.e. w.r.t. Lebesgue measure on $\mathbb{R}^n$. We show in  Appendix under the heading of Claim \ref{cl:cl3},  
that for each $(\hat{x}, \hat{y}) \in \mathcal{D}$, where $\mathcal{D}:=\{ (x,y) \in \mathbb{R}^n \times \mathbb{R}^n : V(x)+V(y) \leq \hat{r}  \}$ and $\hat{r}$ is defined in~\eqref{eq:r}, 
it holds that $\alpha_{(\hat{x}, \hat{y})} >0$. Now, let $(\mathcal{X}, \tau_x)$ and $(\mathcal{Y}, \tau_y)$ be topological vector spaces. Consider the space $\mathcal{X} \times \mathcal{Y}$ equipped with product topology $\tau_{(x,y)}:= \tau\left( \pi_1 ^{-1} (\mathcal{A}) \cap \pi_2 ^{-1} (\mathcal{B}), \forall \mathcal{A} \in \tau_x, \forall \mathcal{B} \in \tau_y\right)$, where $\pi_1(x,y):=x$ and $\pi_2(x,y):=y$; i.e., the smallest topology under which projection maps are continuous. Then on $(\mathcal{X} \times \mathcal{Y},\tau_{(x,y)})$, $V(x)+V(y) \leq \hat{r}$ implies that $V \circ \pi_1(x,y) +V \circ \pi_2(x,y) \leq \hat{r}$. 
As for $\eqref{eq:slcs}$, $\mathcal{X}$, $\mathcal{Y}$ $=\mathbb{R}^n$ and $\tau_x, \tau_y$ coincides with usual metric topology on $\mathbb{R}^n$. $\tau_{(x,y)}$ coincides with usual metric topology on $\mathbb{R}^{2n}$. Since, $V$ is continuous by construction, composition of two continuous functions is  continuous and addition of continuous functions is again continuous. Therefore, $(x,y) \longmapsto V(x)+V(y) $ is a continuous mapping from $(\mathcal{X} \times \mathcal{Y},\tau_{(x,y)})$ to $(\mathbb{R},\tau_{\mathbb{R}})$, where $\tau_{\mathbb{R}}$ is the topology that coincides with the usual metric topology on $\mathbb{R}$. Now, $[0,\hat{r}]$ is compact in $\tau_{\mathbb{R}}$ and preimage of a compact subset of $(\mathbb{R},\tau_{\mathbb{R}})$ under a continuous function is compact (see e.g., ~\cite{folland2013real}). Hence,           $\mathcal{D}$ is compact subset of $(\mathcal{X} \times \mathcal{Y},\tau_{(x,y)})$. Since infimum and supremum are always attained on a compact subset, we have that $\alpha:= \inf_{(x,y) \in \mathcal{D}} \alpha_{(x,y)} >0 $ and $\|P(x,\cdot)-P(y,\cdot)\|_{tv} \leq 2(1-\alpha) $ for all $(x,y) \in \mathcal{D}$; meaning that~\eqref{eq:EIcond2} also holds.
\end{proof}

\section{Non-Asymptotic Analysis of Learning the Expected Reward w.r.t $\nu_{\pi_{}}$}
\label{sec:samples}

Geometric ergodicity of~\eqref{eq:slcs} under the assumptions from Theorem~\ref{thm:t1} allows us to use Birkhoff's pointwise ergodic theorem, from which it holds that
\begin{equation}
\label{eq:t-d}
    \frac{1}{N}  \sum_{i=0}^{N-1} h(x_i) \xrightarrow{a.s} \int_{\mathbb{R}^{n}}h(y) \nu_{\pi} (dy), 
\end{equation}
for any function $h \in \mathcal{L}^{1}(\mathbb{R}^{n}, \mathbb{B}_{\mathbb{R}^{n}}, \nu_{\pi})$, (see e.g., \cite{hairer2006ergodic}). 
Furthermore, note that~\eqref{eq:t11} implies that $ \int_{\mathbb{R}^{n}} V(y) \nu_{\pi}(dy):=\int_{\mathbb{R}^{n}} \|y\|_{2} ^{2} \nu_{\pi}(dy) \leq \frac{n+c\varrho^{2}}{1 - \gamma}$, \cite{hairer2006ergodic}. Hence, for any reward function $r(x)$ that can be written as $(x^{T}\hat{P}x)^q$, 
for some $\hat{P} \succ 0$ and $q \leq 1$, it follows that $ \lim_{N \rightarrow \infty} \frac{1}{N} \sum_{i=0}^{N-1}  r(x_i) \overset{a.s}{=}  \int_{\mathbb{R}^{n}}r(y) \nu_{\pi} (dy) < \infty $ for the SLDS~\eqref{eq:slcs} under the assumptions from Theorem~\ref{thm:t1}. 

Almost sure (a.s) convergence in \eqref{eq:t-d} allows for learning $\mathbb{E}_{ y \thicksim \nu_{\pi}} r(y)$ non-asymptotically from the observable time averages $\frac{1}{N}  \sum_{i=0}^{N-1} r(x_i)$. 
Our objective is to bound $N$ such that given $\epsilon > 0$ and $\delta \in (0,1)$ their exists $N(\epsilon,\delta)$ such that  for all $N \geq N(\epsilon,\delta)$, with probability at least $1-\delta$, it holds~that 
\begin{equation}
\label{eq:t-d-N}
    \left|\frac{1}{N}  \sum_{i=0}^{N-1} r(x_i)  - \int_{\mathbb{R}^{n}}r(y) \nu_{\pi}(dy) \right| \leq \epsilon 
\end{equation}
If $(x_i)_{i \in \mathbb{N}}$ are i.i.d random variables, non-asymptotic bounds for \eqref{eq:t-d-N} follows trivially from Hoeffding's inequality~\cite{boucheron2013concentration,bertail2018new}. This is certainly not the case for samples generated from the SLDS~\eqref{eq:slcs}, which brings us to the theory of regenerative Markov chains~\cite{athreya2014markov}. 
Intuitively, the idea is to append original Markov chain $(x_i)_{i \in \mathbb{N}}$ into an artificial stochastic process $y_i :=(\hat{x_i},\vartheta_i)_{i \in \mathbb{N}}$, such that the conditional distribution of $\hat{x_i}$ given $\hat{x}_{i-1}$  is the same as that $x_i$ given $x_{i-1}$ for each $i \in \mathbb{N}$ and $\vartheta_i$ is a Bernoulli random variable. 

Specifically, let us assume that for Markov chain $(x_i)_{i \in \mathbb{N}}$ there exist a $\beta \in (0,1)$, 
a subset $\mathcal{S} \subset \mathbb{R}^{n}$, and a probability measure $\hat{\nu}$, such that $\inf_{x \in \mathcal{S}} P(x,\cdot) \geq \beta \hat{\nu}(\cdot)$. Then we can define a valid residual kernel $R(x,\cdot)= \frac{P(x,\cdot)-\beta\cdot \chi_{\mathcal{S}}(x)\hat{\nu}(\cdot)}{1-\beta\cdot\chi_{\mathcal{S}}(x) }$. 
Now, conditioned on $\hat{x}_{i-1}=x$, we define $\mathbb{P}\left(\vartheta_{i-1} =1| \hat{x}_{i-1}=x \right)= \beta\cdot \chi_{\mathcal{S}}(x)$; it directly follows that $\mathbb{P}\left(\vartheta_{i-1} =0| \hat{x}_{i-1}=x \right)=1-\beta\cdot \chi_{\mathcal{S}}(x)$. In addition, we define the one-step transition probability from $(\hat{x}_{i-1})$ to $(\hat{x}_i)$ as
\begin{equation}
    \label{eq:reg}
    \mathbb{P}(\hat{x}_i \in \mathcal{A}|\hat{x}_{i-1}=x)= \beta\cdot \chi_{\mathcal{S}}(x)\hat{\nu}(\mathcal{A})+\left( 1-\beta\cdot \chi_{\mathcal{S}}(x)  \right)R(x,\mathcal{A}).
\end{equation}
It is a straightforward check that $\mathbb{P}(\hat{x}_i \in \mathcal{A}|\hat{x}_{i-1}=x)= P(x,\mathcal{A})$. Now, let us define the first regeneration time $T=T_1=\tau:= \tau_{1}= \inf (t \geq 1: \vartheta_{t-1}=1)$, and for each $m \in \mathbb{N}$, $m^\text{th}$ regeneration time $\tau_m := \inf (t > \tau_{m-1}: \vartheta_{t-1}=1)$, $m^\text{th}$ excursion length of the Markov chain as $T_m:= \tau_{m}-\tau_{m-1}$, as well as  $B_{m}:=(\tau_m,\ldots,\tau_{m+1}-1)$ and $x_{B_{m}}=(x_{\tau_m},\ldots, x_{\tau_{m+1}-1})$. If $ \hat{x}_0= x_0 \thicksim \hat{\nu}(\cdot)$, then strong Markov property implies that we can break $(x_i)_{i \in \mathbb{N}}$ into i.i.d blocks process $(x_{{B}_{i}})_{i \in \mathbb{N}}$. Notice that $\mathbb{E}_{\hat{\nu}} T= \mathbb{E}_{\psi} T_m$ for any $m \geq 2$ and any distribution $\psi$ such that $x_0 \thicksim \psi $ (see~\cite{bertail2018new} for more details). Using the Law of Large Numbers on the i.i.d blocks, it can be easily verified (see e.g., \cite{athreya2006measure}) that the invariant measure $\nu_{\pi}$ satisfies that for any $\mathcal{A} \in \mathbb{B}_{\mathbb{R}^{n}}$, it holds that
\begin{equation}
    \label{eq: invm}
    \nu_{\pi}(\mathcal{A}) = \frac{\mathbb{E}_{\hat{\nu}} \sum_{i=0}^{T -1} \chi_{\mathcal{A}} (x_i) }{\mathbb{E}_{\hat{\nu}} T }.
\end{equation}
If we also define $\overline{r}(x_i):= r(x_i)- \mathbb{E}_{y \thicksim \nu_{\pi}} r(y)$, $R(N):= \min(k \geq 1 : \tau_k > N )$ and $\Delta(N):= \tau_{R(N)}-N $, then for each $N \in \mathbb{N}$, it holds that 
\begin{equation}
\label{eq: overs}
\frac{1}{N} \sum_{i=0}^{N-1}  \overline{r}(x_i) = \frac{1}{N} \underbrace{ \sum_{i=0}^{\tau -1} \overline{r}(x_i)}_{\mathcal{O}_1} + \frac{1}{N} \underbrace{\sum_{i= \tau} ^{\tau_{R(N)}-1} \overline{r}(x_i)}_{\mathcal{Z}} - \frac{1}{N} \underbrace{\sum_{i=N}^{\tau_{R(N)}-1} \overline{r}(x_i)}_{\mathcal{O}_2}=:\frac{1}{N}(\mathcal{O}_1 + \mathcal{Z} - \mathcal{O}_2);
\end{equation}
note that here, $\mathcal{Z}$ is essentially the sum of $R(N)-1$  i.i.d blocks . Finally, any non-asymptotic bounds on \eqref{eq:t-d} with $x_0 \thicksim \psi$  would require bounding  $\frac{1}{N} \left( \sqrt{ \mathbb{E}_{\psi} \mathcal{Z}^2} + \sqrt{ \mathbb{E}_{\psi} (\mathcal{O}_1 - \mathcal{O}_2 )^{2} }    \right) $~\cite{latuszynski2013nonasymptotic}.   

To use the aforementioned concept for the SLDS \eqref{eq:slcs}, 
for $\gamma \in (0,1)$, we  define $\hat{V}: \mathbb{R}^n \rightarrow [1, \infty) $ as
\vspace{-2pt}
\begin{flalign}
\label{eq:hatV}
 & \hat{V}(x):= (1+ \frac{1-\gamma}{2}\frac{\|x\|_{2}^{2}}{n}).
\end{flalign}
Now start with the following result.

\begin{theorem}
\label{thm:t2}
An SLDS~\eqref{eq:slcs} under assumptions in Theorem \ref{thm:t1}, with $\mathcal{S}:=\mathcal{B}_{\sqrt{2(n+c\varrho^{2}+1)}}^n=\{x \in \mathbb{R}^n : \|x\|_{2}^{2} \leq 2(n+c\varrho^{2}+1)\}$ and some probability measure $\hat{\nu}$ on $\mathbb{R}^{n}$, satisfies that
\vspace{-2pt}
\begin{equation}
\label{eq:2t}
(\mathcal{P}\hat{V})(x) \leq \lambda\hat{V}(x) + K_2 \chi_{\mathcal{S}}(x), ~~~~\lambda \in (\gamma,1), K_2\in\mathbb{R},
\end{equation} 
\vspace{-8pt}
\begin{equation}
\label{eq:minorize}
\inf_{x \in \mathcal{S}} P(x,\cdot) \geq \beta  \hat{\nu}(\cdot).   
\end{equation}
\end{theorem}

Note that the conditions~\eqref{eq:2t} and~\eqref{eq:minorize} are stronger than the respective conditions~\eqref{eq:EIcond1} and~\eqref{eq:EIcond2}. Since, \eqref{eq:minorize} implies \eqref{eq:EIcond2}, but the other direction does not always hold.
\begin{proof}
For $\|x\|_{2} > \varrho$, we are looking for a $\lambda \in (\gamma,1)$ and a region $\mathcal{C}$ such that 
$\mathcal{P}\hat{V}(x) \leq \lambda \hat{V}(x) $ for all $x \in \mathcal{S}^{c}:= \mathcal{C} \cap (\mathcal{B}_\varrho^n)^\complement$. 
Since (*) below holds from Theorem~\ref{thm:t1} proof, $\mathcal{P}\hat{V}(x) \leq \lambda \hat{V}(x) $ holds if
\begin{flalign}
& \nonumber \mathcal{P}\hat{V}(x) \stackrel{(*)}{\leq} 1+\frac{1-\gamma}{2n}(\gamma\|x\|_{2}^{2}+n) \leq \lambda \left(1+\frac{1-\gamma}{2n}\|x\|_{2}^{2} \right)  \implies (1-\lambda)+ \frac{1-\gamma}{2} \leq \frac{(\lambda-\gamma)(1-\gamma)}{2n}\|x\|_{2}^{2} \\ & \nonumber \implies \left(\frac{3-\gamma}{2} \right) -\lambda \leq \frac{(\lambda-\gamma)(1-\gamma)}{2n}\|x\|_{2}^{2} \quad \implies \quad \|x\|_{2}^{2} \geq \frac{\left(\frac{3-\gamma}{2} \right) -\lambda  }{(\lambda-\gamma)(1-\gamma)}2n > \frac{(1-\lambda)}{(\lambda-\gamma)(1-\gamma)}2n.  
\end{flalign}
Thus, any $\lambda \in (\gamma,1)$, $\mathcal{C}:= \{x \in \mathbb{R}^{n}: \|x\|_{2}^{2} > 2n\}$, $\mathcal{S}^{c}:=\{x \in \mathbb{R}^{n}:\|x\|_{2}^{2} > 2(n+c\varrho^{2}+1)\}$ and $K_2=\frac{3}{2}+2c +c^2\varrho^{2}$ would ensure that~\eqref{eq:2t} holds; this follows  from \eqref{eq:t11} and the fact that $n \geq 1$.

To show that~\eqref{eq:minorize} holds, we extend on the idea of \cite{douc2014long} for the system~\eqref{eq:slcs}. For the compact set $\mathcal{S}$ (from the theorem statement), if $x \in \mathcal{M}_{j} \cap \mathcal{S} $, we have that for $\mathcal{A} \in \mathbb{B}_{\mathbb{R}^{n}}$ it holds that 
\begin{flalign}
& \nonumber P(x,\mathcal{A})= \frac{1}{(2 \pi)^{\frac{n}{2}}} \int_{\mathcal{A}} \exp{\left(\frac{-\|y-\hat{A}_j^\pi x\|_{2}^{2}}{2}\right)}dy \geq \frac{1}{(2 \pi)^{\frac{n}{2}}} \int_{\mathcal{A} \cap \mathcal{S}} \exp{\left(\frac{-\|y-\hat{A}_j^\pi x\|_{2}^{2}}{2}\right)}dy \geq\\ & \nonumber \geq \frac{1}{(2 \pi)^{\frac{n}{2}}}  \inf_{(x,y) \in \mathcal{M}_{j}\cap \mathcal{S} \times \mathcal{S} } \exp{\left(\frac{-\|y-\hat{A}_j^\pi x\|_{2}^{2}}{2}\right)} \lambda^{n}(\mathcal{A} \cap \mathcal{S})  \geq\\ & \label{eq:vol} \geq \frac{1}{(2 \pi)^{\frac{n}{2}}}  \inf_{(x,y) \in \mathcal{M}_{j}\cap \mathcal{S} \times \mathcal{S} } \exp{\left(\frac{-\|y-\hat{A}_j^\pi x\|_{2}^{2}}{2}\right)} \frac{\lambda^{n}(\mathcal{A} \cap \mathcal{S})}{\lambda^{n}(\mathcal{S}) }  \geq\\ &  \label{eq:2t0} \geq \frac{1}{(2 \pi)^{\frac{n}{2}}} \inf_{(x,y) \in \mathcal{S} \times \mathcal{S} } \exp{\left(\frac{-\|y-\hat{A}_j^\pi x\|_{2}^{2}}{2}\right)} \frac{\lambda^{n}(\mathcal{A} \cap \mathcal{S}) }{\lambda^{n}(\mathcal{S}) }   
\end{flalign}
Here,~\eqref{eq:vol} follows from the fact that $\lambda ^{n}(\mathcal{S}) \geq n^{n} \frac{\pi^{\frac{n}{2}}}{\Gamma(1+ \frac{n}{2})} \geq 1$ (see e.g., \cite{folland2013real}), where $\Gamma(1+ \frac{n}{2})= \int_{0}^{\infty} x^{\frac{n}{2}}e^{-x}dx$. Therefore, for any $x \in \mathcal{S}$, we have 
\vspace{-2pt}
\begin{flalign}
& \nonumber P(x,\mathcal{A}) \geq \frac{1}{(2 \pi)^{\frac{n}{2}}} \Bigg( \inf_{(x,y) \in  \mathcal{S} \times \mathcal{S} } \exp{\left(\frac{-\|y-\hat{A}_1^\pi x\|_{2}^{2}}{2}\right)} \land \inf_{(x,y) \in  \mathcal{S} \times \mathcal{S} } \exp{\left(\frac{-\|y-\hat{A}_2^\pi x\|_{2}^{2}}{2}\right)}  \\ & 
\land \ldots 
\label{eq:2t1} \land \inf_{(x,y) \in \mathcal{S}  \times \mathcal{S} } \exp{\left(\frac{-\|y-\hat{A}_M^\pi x\|_{2}^{2}}{2}\right)} \Bigg) \frac{\lambda^{n}(\mathcal{A} \cap \mathcal{S})}{\lambda^{n}(\mathcal{S})}.
\end{flalign}
Therefore, \eqref{eq:minorize} holds with
\begin{equation}
\label{eq:2t2}
\beta=\frac{1}{(2 \pi)^{\frac{n}{2}}} \Bigg( \inf_{(x,y) \in  \mathcal{S} \times \mathcal{S} } \exp{\left(\frac{-\|y-\hat{A}_1^\pi x\|_{2}^{2}}{2}\right)}  \land \ldots  
\land \inf_{(x,y) \in  \mathcal{S}  \times \mathcal{S} } \exp{\left(\frac{-\|y-\hat{A}_M^\pi x\|_{2}^{2}}{2}\right)} \Bigg),
\end{equation}
where compactness of $\mathcal{S} \times \mathcal{S} $ in the product topology and finitely many `$\land$' operations ensures that $\beta \in (0,1)$, 
and
$\hat{\nu}(\cdot)=  \frac{\lambda^{n}(\cdot \cap \mathcal{S})}{ \lambda^{n}(\mathcal{S}) }.$
\end{proof}

\begin{theorem}
\label{thm:thm3}
Consider a (fixed) $\delta \in (0,1)$ and any reward function of the form $r(x)=\sqrt{x^{T} \hat{P} x}$, where $\hat{P} \succ 0$. If N satisfies $N \geq \Omega(\frac{n}{\beta \epsilon^2 \delta (1-\gamma)})$, then~\eqref{eq:t-d-N} holds with probability at least $1-\delta$. 
\end{theorem}

Before moving to the proof,
which is based on the approach from~\cite{latuszynski2013nonasymptotic},\footnote{As we rely on the approach and proof from~\cite{latuszynski2013nonasymptotic}, and apply it to the SLDS from \eqref{eq:slcs}, we employ the same notation as~\cite{latuszynski2013nonasymptotic}.} 
we define as well as upper bound  
$\pi(\hat{V}):= \int \hat{V}(y) \nu_{\pi} (dy) =1 + \frac{1-\gamma}{2n}\int V(y) \nu_{\pi} (dy)  \leq 1 +  \frac{(1-\gamma)}{2n} \frac{(n+c\varrho^{2})}{(1-\gamma)} = \frac{3}{2} +\frac{c\varrho^{2}}{2n}$,  $\mathcal{E}_{x}(\hat{V}) := \sup_{N \in \mathbb{N}} \mathcal{P}^{N} \hat{V}(x)= 1+ (\frac{1-\gamma}{2n}) \sup_{N \in \mathbb{N}} \mathcal{P}^{N}V(x) \leq \pi(\hat{V}) +\frac{(1-\gamma)}{2n} \gamma \|x\|_{2}^{2}$. W.l.o.g assume  $\|\hat{P}\|_{2} \leq 1$, then $\|\overline{r}\|_{\hat{V}^{\frac{1}{2}}}^2 := \sup_{x \in \mathbb{R}^{n}}  \left(\frac{|\overline{r}(x)|}{\hat{V}^{\frac{1}{2}}(x)}\right)^{2}  \leq       \sup_{x \in \mathbb{R}^{n}} \frac{x^{T}\hat{P}x+\pi(\hat{V})}{1 + \frac{(1-\gamma)}{2n}\|x\|_{2}^{2}} \leq \frac{2n}{(1-\gamma)}$. With these upper bounds and values $\beta, K_2, \lambda $  we can also compute upper bounds on the following variables (which allows us to capture major bound terms) $\sigma_{as}^{2}(P,r):= \frac{\mathbb{E}_{\hat{\nu}} \left( \sum_{i=0}^{T-1} \overline{r}(x_i)\right)^2}{ \mathbb{E}_{\hat{\nu}} T}$, $C_{0}(P):= \mathbb{E}_{\nu_{\pi}}T- \frac{1}{2}$, $C_1(P,r):= \sqrt{\mathbb{E}_{x} \left(  \sum_{i=0}^{T-1} |\overline{r}|(x_i)  \right)^{2}}$ and $C_2(P,r):= \sqrt{\mathbb{E}_{x} \left( \chi_{(T<n)} \sum_{i=n}^{\tau_{R(n)}-1}  |\overline{r}|(x_i)\right)^{2}}$ -- this directly follows from Theorem 4.2 and Proposition 4.5 from ~\cite{latuszynski2013nonasymptotic}, which are satisfied if the Markov chain satisfies \eqref{eq:2t} and \eqref{eq:minorize}. Details of the proof is given in Appendix.

\subsection{Discussion: I.I.D Block Sequences and Their Link to Sample Complexity.} 

Although Theorem~\ref{thm:thm3}, gives explicit bounds on sample complexity by computing upper bounds on specific coefficients related to SLDS governed by \eqref{eq:slcs}, using Theorem 4.2 and Proposition 4.5 from ~\cite{latuszynski2013nonasymptotic} and then sample complexity follows from Theorem 3.1 of ~\cite{latuszynski2013nonasymptotic}. However, to give a clear explanation to the reader of the how does the existence of i.i.d blocks for Markov chain from SLDS in \eqref{eq:slcs} leads to the sample complexity shown in Theorem \ref{thm:thm3}, recall:
\begin{equation}
\label{eq:block}
\mathbb{P}_{\psi}(|\frac{1}{N}  \sum_{i=0}^{N-1} r(x_i)  - \int_{\mathbb{R}^{n}}r(y) \nu_{\pi} (dy)| > \epsilon ) \leq \frac{ \mathbb{E}_{\psi} (\mathcal{Z}^{2}+ (\mathcal{O}_1 -\mathcal{O}_2 )^2 +2 \mathcal{Z}(\mathcal{O}_1 -\mathcal{O}_2)  }{N^2\epsilon ^{2} }
\end{equation}
As we already have a lower bound on the sample complexity in Theorem \ref{thm:thm3}, it is sufficient to consider bounding $\frac{1}{N^2 \epsilon^2}\mathbb{E}_{\psi} \mathcal{Z}^2 $ from \eqref{eq:block}, with $x_0 \thicksim \psi$ for any arbitrary distribution $\psi$. Recall that $\mathcal{Z}$ corresponds to i.i.d block sequences. To proceed, we first bound the term $\mathbb{E}_{\hat{\nu}} \left(\sum_{i=0}^{\tau_{R(N)}-1} \overline{r}(x_i)\right)^{2}$ as 
\begin{flalign}
 & \label{eq:asyvar} \mathbb{E}_{\hat{\nu}} \left(\sum_{i=0}^{\tau_{R(N)}-1} \overline{r}(x_i)\right)^{2} \stackrel{(blocks)}{=} \mathbb{E}_{\hat{\nu}} \left(\sum_{m=1}^{R(N)} \overline{r}(x_{B_m})\right)^{2} \stackrel{(**)}{=} \mathbb{E}_{\hat{\nu}} \left(\sum_{i=0}^{T-1} \overline{r}(x_i)\right)^{2} \mathbb{E}_{\hat{\nu}}R(N) = \\  & \nonumber = \underbrace{\frac{\mathbb{E}_{\hat{\nu}} \left(\sum_{i=0}^{T-1} \overline{r}(x_i)\right)^{2}}{\mathbb{E}_{\hat{\nu}} T}}_{\sigma_{as}^{2}(P,r)} \mathbb{E}_{\hat{\nu}} T \mathbb{E}_{\hat{\nu}}R(N) = \sigma_{as}^{2}(P,r) \mathbb{E}_{\hat{\nu}} \tau_{R(N)}= \sigma_{as}^{2}(P,r)\cdot \underbrace{(N+ \mathbb{E}_{\hat{\nu}} \Delta(N))}_{\stackrel{(***)}{\leq} N+ C^1(\lambda) \pi(\hat{V})+ \frac{1}{\beta}C^2(\lambda,K_2)},
\end{flalign}
where (blocks) converts summation over trajectory into summation over the $R(N)$ i.i.d blocks, and $(**)$ follows by applying Wald's Lemma on the i.i.d blocks. In addition, $(***)$ holds for sufficiently large $\beta$ from Theorem~4.2 in~\cite{latuszynski2013nonasymptotic}, and $C^{1}(\cdot)$ and $C^{2}(\cdot)$ are constant functions of their respective~arguments.

Now, for $\sigma_{as}^{2}(P,r)$ defined as in~\eqref{eq:asyvar}, it holds that
\begin{equation*}
\begin{split} 
\label{eq:asyvar2}   \sigma_{as}^{2}(P,r):&= \frac{\mathbb{E}_{\hat{\nu}} \left(\sum_{i=0}^{T-1} \overline{r}(x_i)\right)^{2}}{\mathbb{E}_{\hat{\nu}} T} = \frac{\mathbb{E}_{\hat{\nu}} \left(\sum_{i=0}^{T-1}\overline{r}^{2}(x_i)\right) }{\mathbb{E}_{\hat{\nu}} T} + \frac{\mathbb{E}_{\hat{\nu}} \left(\sum_{i=0}^{T-1} \overline{r}(x_i) ( \sum_{j \neq i}^{T-1} \overline{r}(x_j) ) \right) }{\mathbb{E}_{\hat{\nu}} T}= \\ & 
\stackrel{(i)}{=}\mathbb{E}_{\nu_{\pi}} \overline{r}^{2} + 2\mathbb{E}_{\nu_{\pi}}\left( \sum_{i=1}^{T-1} \overline{r} P^{i}\overline{r} \right)   \stackrel{(ii)}{\leq} \frac{2n}{1-\gamma} \underbrace{\left(  \mathbb{E}_{\nu_{\pi}}\hat{V} +
2\mathbb{E}_{\nu_{\pi}}\left( \sum_{i=1}^{T-1} \hat{V}^{\frac{1}{2}} P^{i}\hat{V}^{\frac{1}{2}} \right) \right)}_{ \stackrel{(iii)}{\leq} \frac{1}{\beta}C^3(K_2,\lambda) \pi(\hat{V})}.
\end{split}
\end{equation*}
Here, $(i)$ follows from the structure of the invariant measure~\eqref{eq: invm} and stationarity, whereas $(ii)$ holds because $\| \overline{r} \|_{\hat{V}^\frac{1}{2}}^2 \leq \frac{2n}{1-\gamma}$. 
In addition, $(iii)$ holds from Theorem~4.2 in~\cite{latuszynski2013nonasymptotic}, and $C^{3}(\cdot)$ is a constant function of its argument. From Theorem \ref{thm:thm3}, we have that $\pi(\hat{V}) \leq \frac{3}{2} (1+ \frac{c\varrho^2}{n}) \leq \frac{3}{2} (1+ c\varrho^2)$, as $n \geq 1$. Therefore, $\mathbb{E}_{\hat{\nu}} \left(\sum_{i=0}^{\tau_{R(N)}-1} \overline{r}(x_i)\right)^{2} \leq \frac{nN}{(1-\gamma)\beta} C^{4}(K_2,\lambda,c,\varrho^2)$, where $C^{4}(K_2,\lambda,c,\varrho^2)$ is a constant function depending on $K_2, \lambda,c$ and $\varrho^2$.
With this inequality at hand, if $x_0 \thicksim \psi$ then
as captured in~\cite{latuszynski2013nonasymptotic} (discussion between (3.10) and (3.11) in~\cite{latuszynski2013nonasymptotic})
it holds that
\begin{flalign}
& \nonumber \frac{1}{N^2}\mathbb{E}_{\psi} \mathcal{Z}^2=
\frac{1}{N^2} \mathbb{E}_{\psi} \left(\sum_{i=\tau}^{\tau_{R(N)}-1} \overline{r}(x_i)\right)^2 = \frac{1}{N^2} \sum_{j=1}^{N} \mathbb{E}_{\psi}\bigg( \left(\sum_{i=\tau}^{\tau_{R(N)}-1}  \overline{r}(x_i)\right)^2| T=j\bigg) \mathbb{P}_{\psi}(T=j) \\ & \label{eq:final} = \frac{1}{N^2} \sum_{j=1}^{N} \mathbb{E}_{\hat{\nu}}  \left(\sum_{i=0}^{\tau_{R(N-j)}-1}  \overline{r}(x_i)\right)^2 \mathbb{P}_{\psi}(T=j) \leq \frac{n}{N(1-\gamma)\beta} C^{4}(K_2,\lambda,c,\varrho^2).
\end{flalign}
Now, it directly follows that the sample complexity is $\Omega(\frac{n}{\beta \epsilon^2 \delta (1-\gamma)})$.
\section{Case Studies}
\label{sec:case}
One of our main results is proving linear dependence of the sample complexity on the state space dimensions (with all other variables fixed). 
To validate this phenomenon empirically, we generate a sequence of closed-loop matrices $(\hat{A}_{1}^{\pi} (n),\hat{A}_{2}^{\pi} (n))_{n \in \mathbb{N}}$ such that $\hat{A}_{1}^{\pi} (n) =\gamma I_n$, and $\hat{A}_{2}^{\pi} (n)=c I_n$, where we assign $\gamma=0.9$ and $c=2$,
as well as $\varrho:=10$ (as defined in Theorem~\ref{thm:t1}). 
Specifically, we consider the SLDSs \eqref{eq:slcs}, with size of the state space $n$ varying from $n=1:2000$, with increments of 50, and $M=2$ state regions, resulting in 
\begin{align*}
x_{t+1} (n)&=\hat{A}_{1}^{\pi} (n)x_{t}(n) +w_{t}^1(n), \qquad \text{if } \|x_{t}(n)\|>\varrho \\ x_{t+1}(n)&=\hat{A}_{2}^{\pi} (n) x_{t}(n)+w_{t}^2(n), \qquad \text{ if }  \|x_{t}(n)\|\leq \varrho,
\end{align*}
where $w_{t}^1(n)$ and $w_{t}^2(n)$ are appropriate $n$ dimensional Gaussians as discussed in Section~\ref{sec:mixing}.
With an accuracy parameter $\epsilon:=1e-10$ (from~\eqref{eq:t-d-N}) we define our pseudo sample complexity for state space of size $n$  as the smallest $N(n) \in \mathbb{N}$ such that \begin{equation}
\left|\frac{1}{N(n)} \sum_{t=0}^{N(n)}r(x_{t}(n))-\frac{1}{N(n)+1} \sum_{t=0}^{N(n)+1}r(x_{t}(n))\right| < \epsilon.
\end{equation}

We simulated 100000 independent trials for every considered size of the state space varying from $n=1:2000$ and averaged the pseudo sample complexity for a more accurate description i.e, $N^{a}(n) := \frac{\sum_{i=1}^{100000} N_{i}(n) }{100000}$, where `$i$' represents each independent trial. As shown in Figure~\ref{fig:2figsA}, in higher dimensions sample complexity depends linearly on dimensions of the state space. 
Another important factor is the dependence on `$\gamma$'. We repeated the aforementioned procedure with $10000$ independent trials and same system dynamics, but with `$\gamma$'  varied from $0.5$ to $0.9$ with increments of $0.05$; the obtained results are shown in Figure~\ref{fig:2figsB}. Our results from Figure~\ref{fig:2figsB} validate that the sample complexity degrades with an increase in $\gamma$ as captured in Theorem~\ref{thm:thm3}.

\begin{figure}[!t]
\centering
\parbox{6.6cm}{
\includegraphics[width=6.6cm]{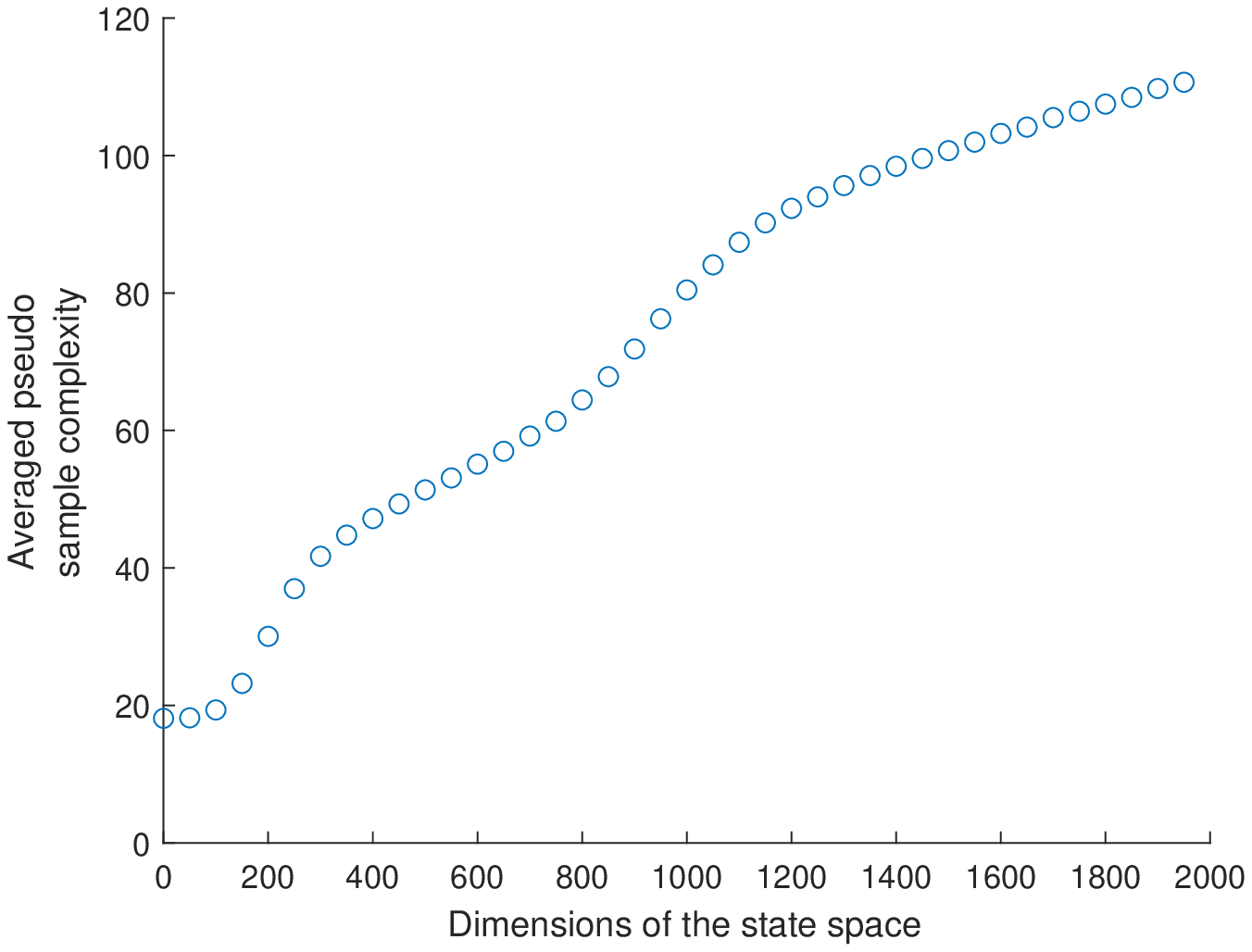}
\vspace{-12pt}
\caption{Average pseudo-sample complexity v.s. the state-space dimension.}
\label{fig:2figsA}}
\hspace{6pt}
\begin{minipage}{6.8cm}
\includegraphics[width=6.8cm,height=4.0cm]{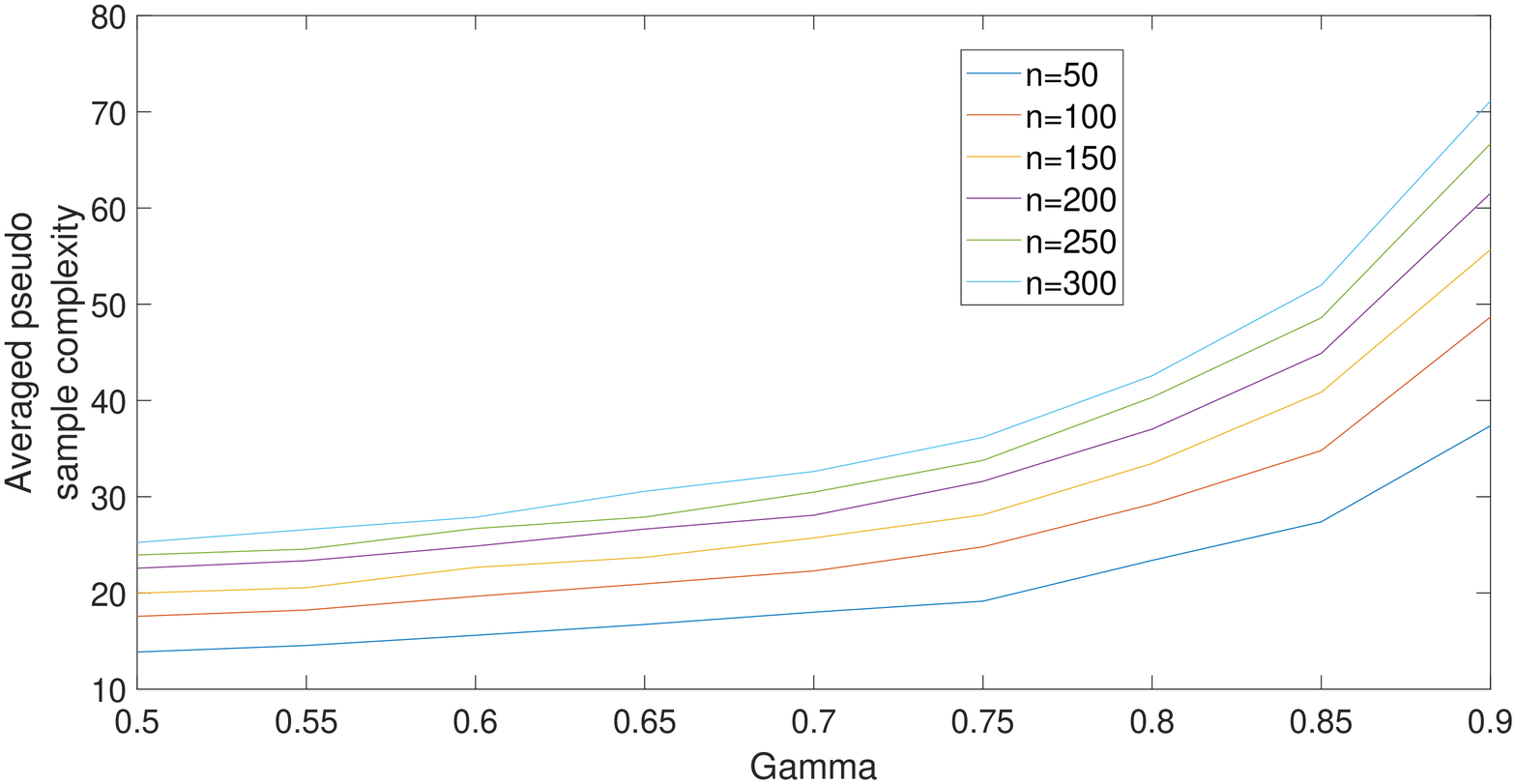}
\caption{Average pseudo-sample complexity v.s. `$\gamma$, for varying state-space dimension}
\label{fig:2figsB}
\end{minipage}
\end{figure}

\section{Conclusion}
\label{sec:conclusion}

We showed existence of invariant ergodic measure for closed-loop switched linear dynamical systems, which are stable in an unbounded subset of the state-space. In addition, we derived non-asymptotic bounds for learning the expected reward from time-averages. 
With all other parameters fixed, we showed that the  sample complexity of learning the expected reward (w.r.t the ergodic invariant measure the closed-loop switched linear dynamical systems mixes to) is linear to the state-space size and inverse quadratic in the approximation error $\epsilon$ (i.e.,~$\Omega(\frac{n}{\epsilon^2})$); hence, extending existing  non-asymptotic results to a class of nonlinear dynamical systems. By learning the expected reward instead of a value function parameterized by a discount factor, we provided a non-asymptotic analysis that is  valid for applications that require minimizing asymptotic~rewards.

%
\section{Acknowledgement}
This work is sponsored in part by the AFOSR under award number FA9550-19-1-0169, as well as the NSF CNS-1652544 and CNS-1837499 awards.
\bibliography{neurips_2020}
\bibliographystyle{dinat}
\newpage
\section*{Appendix}
\label{sec:appendix}

\begin{claim}
\label{cl:cl3}
For every $(x,y)\in \mathcal{D}$, where $\mathcal{D}:=\{ (x,y) \in \mathbb{R}^n \times \mathbb{R}^n : V(x)+V(y) \leq \hat{r}  \}$ and $\hat{r}$ is defined in  ~\eqref{eq:r},
it holds that an $\alpha_{(x,y)}:=\int (f_{x}(z) \land g_{y}(z))dz  >0$.
\end{claim}

\begin{proof}
As discussed in Section  ~\ref{sec:mixing}
,it suffices to prove that  $ f_{x}(z) \land g_{y}(z) >0 $ a.e. w.r.t Lebesgue measure on $\mathbb{R}^{n}$. 

We identify the following three cases and prove the claim for each case. 

\vspace{10pt}\textbf{Case 1.} $\|x\|_{2} \leq \varrho$, $\|y\|_{2} > \varrho$ 
and without loss of generality (w.l.o.g.) assume that $P(x,\cdot) \thicksim \mathcal{N}(\hat{A}_j^\pi x, {I}_n)$ and $P(y,\cdot) \thicksim \mathcal{N}(\hat{A}_k^\pi x, {I}_n)$ where $\|\hat{A}_j^\pi\|_{2}^{2} \leq c$ and $\|\hat{A}_k^\pi\|_{2}^{2} \leq \gamma <1$. Then it holds that 
    \begin{flalign}
    \label{eq:C1}
    \nonumber  f_{x}(z)\land g_{y}(z) &=\\ \nonumber 
    =\frac{1}{(2 \pi)^{\frac{n}{2}}}&\exp( - \frac{\|z\|_2 ^{2}}{2})\Bigg( \exp(-\frac{\|\hat{A}_j^\pi x\|_2^{2}}{2})\exp(-z^{T}\hat{A}_j^\pi x) \land  \exp{(-\frac{\|\hat{A}_k^\pi y\|_2^{2}}{2})}\exp(-z^{T}\hat{A}_k^\pi y)\Bigg)  \\\nonumber
    >
    \frac{1}{(2 \pi)^{\frac{n}{2}}}&\exp( - \frac{\|z\|_2 ^{2}}{2}) \Bigg(\exp(-\frac{\|\hat{A}_j^\pi x\|_2^{2}}{2})\exp(-z^{T}\hat{A}_j^\pi x) \land    \exp(-\frac{\| y\|_2^{2}}{2})\exp(-z^{T}\hat{A}_k^\pi y)\Bigg)  \\ \nonumber
    \geq \frac{1}{(2 \pi)^{\frac{n}{2}}}&\exp( - \frac{\|z\|_2 ^{2}}{2}) \Bigg( \exp(-\frac{\|\hat{A}_j^\pi\|_2 ^{2}\| x\|_2^{2}}{2})\exp(-z^{T}\hat{A}_j^\pi x) \land   \exp(-\frac{\| y\|_2^{2}}{2})\exp(-z^{T}\hat{A}_k^\pi y) \Bigg)  \\ \nonumber
    = \frac{1}{(2 \pi)^{\frac{n}{2}}}&\exp(-\frac{\|z\|_2 ^{2}}{2}) \exp(-\frac{\|\hat{A}_j^\pi\|_2 ^{2}\| x\|_2^{2}}{2}) \Bigg(\exp(-z^{T}\hat{A}_j^\pi x) \land  \\ & \nonumber \exp(\frac{\|\hat{A}_j^\pi\|_2 ^{2}\| x\|_2^{2}}{2}) \exp(-\frac{\| y\|_2^{2}}{2})\exp(-z^{T}\hat{A}_k^\pi y)\Bigg)  \\ \nonumber \geq \frac{1}{(2 \pi)^{\frac{n}{2}}}&\exp( - \frac{\|z\|_2 ^{2}}{2}) \exp(-\frac{\|\hat{A}_j^\pi\|_2 ^{2}\| x\|_2^{2}}{2}) \Bigg( \exp(-z^{T}\hat{A}_j^\pi x) \land  \\ &  \exp(\frac{\|\hat{A}_j^\pi\|_2 ^{2}\| x\|_2^{2}}{2}) \exp(-\frac{(n+ c \varrho^{2})}{ \gamma (1-\gamma )}  )\exp(-z^{T}\hat{A}_k^\pi y) \Bigg), 
    \end{flalign}
where the last inequality follows from the fact that $ \frac{2(n+ c \varrho^{2})}{ \gamma (1-\gamma )} \geq \|y\|_2^{2} > \varrho^{2} $. Since $(1-\gamma)>0$, we get that the right side of \eqref{eq:C1} is $ > 0$ a.e w.r.t Lebesgue measure on $\mathbb{R}^n$.
%

\vspace{10pt}\textbf{Case 2.} $\|x\|_{2} > \varrho$, $\|y\|_{2} > \varrho$ and
    w.l.o.g let us assume that $P(x,\cdot) \thicksim \mathcal{N}(\hat{A}_l^\pi x, {I})$ and $P(y,\cdot) \thicksim \mathcal{N}(\hat{A}_m^\pi x, {I})$, where  $\|\hat{A}_l^\pi\|_{2}^{2} \leq \gamma$ and
    $\|\hat{A}_m^\pi\|_{2}^{2} \leq \gamma$. Then, the following holds:
    \begin{flalign}
    \label{eq:C2}
    \nonumber f_{x}(z)\land g_{y}(z) & = \\ \nonumber
    =\frac{1}{(2 \pi)^ \frac{n}{2}} &\exp(-\frac{\|z\|_{2}^2}{2})\Bigg( \exp(- \frac{\|\hat{A}_l^\pi x\|_{2}^{2}}{2}) \exp(-z^{T} \hat{A}_l^\pi x ) \land \exp(- \frac{\|\hat{A}_m^\pi y\|_{2}^{2}}{2}) \exp(-z^{T} \hat{A}_m^\pi y )\Bigg)  \\ \nonumber 
    \geq \frac{1}{(2 \pi)^ \frac{n}{2}} &\exp(-\frac{\|z\|_{2}^2}{2})\Bigg( \exp(- \gamma \frac{\|x\|_{2}^{2}}{2}) \exp(-z^{T} \hat{A}_l^\pi x ) \land \exp(- \gamma \frac{\|y\|_{2}^{2}}{2}) \exp(-z^{T} \hat{A}_m^\pi y )\Bigg)  \\ 
    \geq \frac{1}{(2 \pi)^{\frac{n}{2}}}& \exp(-\frac{\|z\|_{2}^{2}}{2})\exp(- \frac{n+c\varrho^{2}}{(1-\gamma)}) \Bigg( \exp(-z^{T}\hat{A}_l^\pi x) \land \exp(-z^{T}\hat{A}_m^\pi x) \Bigg).  
    \end{flalign}
    Here, the last inequality follows from the fact that $ \frac{2(n+ c \varrho^{2})}{ \gamma (1-\gamma )} \geq \|y\|_2^{2} > \varrho^{2} $ and  $ \frac{2(n+ c \varrho^{2})}{ \gamma (1-\gamma )} \geq  \|x\|_{2}^{2} > \varrho^{2}$.
    Since, $(1-\gamma)>0$, it holds that the right side of \eqref{eq:C2} is $>0$ a.e w.r.t Lebesgue measure on $\mathbb{R}^n$.
    %

\textbf{Case 3.} $\|x\|_{2} \leq \varrho$, 
$\|y\|_{2} \leq \varrho$ and w.l.o.g we assume that $P(x,\cdot) \thicksim \mathcal{N}(\hat{A}_u^\pi x, {I})$ and $P(y,\cdot) \thicksim \mathcal{N}(\hat{A}_o^\pi x, {I})$, where  $\|\hat{A}_u^\pi\|_{2}^{2} \leq C$ and
    $\|\hat{A}_o^\pi\|_{2}^{2} \leq C$. Then, it holds that 
    \begin{flalign}
    \nonumber f_{x}(z)\land g_{y}(z) &= \\ \nonumber 
    = \frac{1}{(2 \pi)^{\frac{n}{2}}} &\exp(- \frac{\|z\|_2 ^{2}}{2}) \Bigg ( \exp(-\frac{\|\hat{A}_u^\pi x\|_{2}^{2}}{2})\exp(-z^{T}\hat{A}_u^\pi x) \land  \exp(-\frac{\|\hat{A}_o^\pi x\|_{2}^{2}}{2})\exp(-z^{T}\hat{A}_o^\pi x) \Bigg)  \\ \nonumber \geq \frac{1}{(2 \pi)^{\frac{n}{2}}}&\exp(-\frac{\|z\|_{2}^{2}}{2}) \Bigg(\exp(-\frac{C\|x\|_{2}^{2}}{2}) \exp(-z^{T}\hat{A}_u^\pi x) \land  \exp(-\frac{C\|y\|_{2}^{2}}{2}) \exp(-z^{T}\hat{A}_o^\pi) \Bigg) 
    \\ \label{eq:C3} 
    \geq \frac{1}{(2 \pi)^{\frac{n}{2}}} &\exp(- \frac{\|z\|_{2}^{2}}{2}) \exp(- \frac{C\varrho^{2}}{2})\Bigg(\exp(-z^{T}\hat{A}_u^\pi x) \land  \exp(-z^{T}\hat{A}_o^\pi y) \Bigg). 
    \end{flalign}
    %

Hence, the right side of~\eqref{eq:C3} is $>0$ a.e w.r.t Lebesgue measure on $\mathbb{R}^{n}$, which concludes the proof.
\end{proof}

\vspace{12pt}
\subsection*{Proof of Theorem \ref{thm:thm3}}
\begin{proof}
Using Theorem~3.1 in~\cite{latuszynski2013nonasymptotic} we get that
\begin{flalign}
    & \nonumber \mathbb{E}_{x} (\frac{1}{N} \sum_{i=0}^{N-1} (r(x_i) - \int r(y) \nu_{\pi}(dy)))^{2} \\  &  \leq \frac{\sigma_{as}^{2}(P,r)}{N}(1+ \frac{C_{0}^{2}(P)}{N^{2}}+ 2 \frac{C_{0}(P)}{N}) + 4\frac{C_1(P,r)}{N} \frac{\sigma_{as}(P,r)}{\sqrt{N}} (1 + \frac{C_0(P)}{N})+ 4 \frac{C_{1}^{2}(P,r)}{N^{2}}.
\end{flalign}
Now, from Theorem 4.2 and Proposition 4.5 in~\cite{latuszynski2013nonasymptotic}, it follows that
\begin{equation}
\label{eq:p26}
4\frac{C_1(P,r)}{N} \frac{\sigma_{as}(P,r)}{\sqrt{N}} (1 + \frac{C_0(P)}{N})  \leq 4 \frac{c_{10as}}{N\sqrt{N}}(2\pi(\hat{V})+\frac{1-\gamma}{2n}\gamma \|x\|_{2}^{2}) \frac{2n}{1-\gamma}   \leq  4 (\frac{c_{10as}}{N}) (\frac{6n + 2c\varrho^{2} + \gamma \|x\|_{2}^{2}}{1-\gamma}),
\end{equation}
\begin{flalign}
\label{eq:p27}
 4 \frac{C_{1} ^{2}(P,r)}{N^{2}} \leq 4 (\frac{c_1^{2}}{N^{2}}) (\frac{3n+c\varrho^{2}+ \gamma (1-\gamma)\|x\|_{2}^{2}}{1-\gamma}),
\end{flalign}
\begin{flalign}
\label{eq:p28}
 2 \sigma_{as}^{2} (P,r) \frac{C_0 (P)}{N^{2}}  \leq c_{2as0} \frac{\frac{18}{4}n +c^2\varrho^{4} +3c\varrho^{2}}{N^{2}(1-\gamma)}, 
\end{flalign}
\begin{flalign}
\label{eq:p29}
 \frac{\sigma_{as}^{2} (P,r)}{N} \frac{C_{0}^2(P)}{N^{2}} \leq c_{2as20} \frac{ \frac{27}{4}n + \frac{3c^{2}\varrho^{2}}{4} +\frac{27c\varrho^{2}}{4} + \frac{c^3 \varrho^{4} }{4} + \frac{3c^{2}\varrho^{4}}{2} }{(1-\gamma)  N^{3}},
\end{flalign}
\begin{flalign}
& \frac{\sigma_{as}^{2} (P,r)}{N} \leq 
(\frac{1+\sqrt{\lambda}}{1-\sqrt{\lambda}}+ \frac{2(\sqrt{K_2} -\sqrt{\lambda} -\beta)}{\beta(1-\sqrt{\lambda})})\frac{2n}{N(1-\gamma)} \pi(\hat{V}) \label{eq:p30}  \leq \frac{4(1+ \sqrt{\gamma}\sqrt{1-\gamma}\sqrt{2+c\varrho^{2}})}{N\beta(1-\gamma)^{2}}(3n+c\varrho^{2})
\end{flalign}
Therefore, it holds that
\begin{flalign}
& \nonumber \mathbb{P}(|\frac{1}{N}  \sum_{i=0}^{N-1} r(x_i)  - \int_{\mathbb{R}^{n}}r(y) \nu_{\pi} (dy)| > \epsilon ) \leq \delta \implies \mathbb{E}_{x}\frac{(| \frac{1}{N} \sum_{i=0}^{N-1} (r(x_i) - \nu_{\pi}(r)) |^{2}) }{\epsilon ^{2} } \leq \delta \\ & \nonumber \implies \frac{\frac{\sigma_{as}^{2}(P,r)}{N}(1+ \frac{C_{0}^{2}(P)}{N^{2}}+ 2 \frac{C_{0}P}{N}) + 4\frac{C_1(P,r)}{N} \frac{\sigma_{as}(P,r)}{\sqrt{N}} (1 + \frac{C_0(P)}{N})+ 4 \frac{C_{1}^{2}(P,r)}{N^{2}}}{\epsilon ^{2}} \leq \delta \\ &
\label{eq:smpl}
\implies N \geq \frac{o_1(o_2n+o_3+\gamma \|x\|_{2}^{2}) }{(1-\gamma) \delta \beta \epsilon ^{2}}, 
\end{flalign}
where \eqref{eq:smpl} follows from~\eqref{eq:p26}-\eqref{eq:p30}. 
$c_{10as}$, $c_{1}^{2}$, $c_{2as0}$ and $c_{2as20}$ are constants that contain $\lambda$ factor, but since we left $\lambda$ as any arbitrary value between $(\gamma,1)$, we ignore writing down the tedious exact form the aforementioned constants. However, $o_1$, $o_2$ and $o_3$ are fixed constants independent of $n$, $\gamma$ and~$\lambda$.
\end{proof}

\end{document}